\documentclass[11pt,a4paper,reqno]{amsart}

\usepackage[T1]{fontenc}
\usepackage{lmodern}
\usepackage[margin=30mm]{geometry}
\usepackage{mathtools,amssymb,booktabs,enumitem,microtype}
\usepackage[hidelinks]{hyperref}
\hypersetup{
  pdftitle={The strong nonfinite basis property of S7},
  pdfauthor={Lili Wang, Qingrui Yin and Aifa Wang},
  pdfsubject={An affirmative answer to Problem 1.2 of Yue, Ren and Gao},
  pdfkeywords={additively idempotent semiring, finite basis problem, S7, semilinear set}
}

\newtheorem{theorem}{Theorem}[section]
\newtheorem{lemma}[theorem]{Lemma}
\newtheorem{proposition}[theorem]{Proposition}
\newtheorem{corollary}[theorem]{Corollary}
\theoremstyle{definition}

\theoremstyle{remark}
\newtheorem{remark}[theorem]{Remark}
\numberwithin{equation}{section}

\newcommand{\N}{\mathbb N}
\newcommand{\V}{\mathsf V}
\newcommand{\F}{\mathcal F}
\newcommand{\E}{\mathcal E}
\newcommand{\Pow}{\mathcal P}
\newcommand{\Id}{\operatorname{Id}}
\newcommand{\Wd}{\operatorname{Wd}}
\newcommand{\cm}{\operatorname{cm}}
\newcommand{\supp}{\operatorname{supp}}
\newcommand{\occ}{\operatorname{occ}}
\newcommand{\Sym}{\operatorname{Sym}}
\newcommand{\nnorm}[1]{\lVert #1\rVert_1}
\newcommand{\vect}[1]{\mathbf #1}
\newcommand{\ssev}{S_7}
\setlist[enumerate]{label=\textup{(\roman*)},leftmargin=2.2em,itemsep=3pt}
\allowdisplaybreaks[1]
\title[The strong nonfinite basis property of $S_7$]{The strong nonfinite basis property of $S_7$}
\author{Lili Wang}
\address{School of Mathematical Sciences, Chongqing University of Technology,
Chongqing 400054, P.R. China}
\email{wllaf@cqut.edu.cn}
\author{Qingrui Yin}
\address{School of Mathematical Sciences, Chongqing University of Technology,
Chongqing 400054, P.R. China}
\email{13996559059@163.com}
\author{Aifa Wang}
\address{School of Mathematical Sciences, Chongqing University of Technology,
Chongqing 400054, P.R. China}
\email{wangaf@cqut.edu.cn}
\thanks{Corresponding author: Aifa Wang.}
\thanks{This work was supported by the Science and Technology Research Program of Chongqing
Municipal Education Commission (KJZD-K2024011102) and the Chongqing Natural
Science Foundation Innovation and Development Joint Fund (Municipal Education
Commission) (CSTB2025NSCQ-LZX0067).}
\date{}
\subjclass[2020]{16Y60, 03C05, 08B05, 08B15}
\keywords{Additively idempotent semiring, variety, identity, finite basis problem, semilinear set}

\begin{document}

\begin{abstract}
We prove that the three-element additively idempotent semiring $S_7$ is
strongly nonfinitely based: every finite additively idempotent semiring
whose generated variety contains $S_7$ has no finite equational basis.
This answers Problem~1.2 in Yue, Ren and Gao's paper \emph{Explicit
equational bases for the power semirings of $S_7$} affirmatively.
The proof combines a rigidity property of set systems with private
points and a probabilistic construction derived from semilinear sets
of letter counts in finite semigroups. The argument allows arbitrary
polynomial substitutions and does not assume commutative multiplication.
As consequences, both power semirings of the multiplicative reduct of
$S_7$, and both power semirings of every finite group of order at least
three, are strongly nonfinitely based. All algebras are considered in
the signature $(+,\cdot)$ without nullary operations.
\end{abstract}

\maketitle

\section{Introduction}

An \emph{additively idempotent semiring}, or \emph{ai-semiring}, is an
algebra $(R,+,\cdot)$ such that $(R,+)$ is a commutative idempotent
semigroup, $(R,\cdot)$ is a semigroup, and multiplication distributes
over addition on both sides. The finite basis problem asks which
algebras in a given class admit a finite basis for their identities.
For an algebra $A$, let $\V(A)$ denote the variety generated by $A$.
Thus $A$ is \emph{finitely based} if $\V(A)$ can be defined by finitely
many identities, and is \emph{nonfinitely based} otherwise.

The ai-semiring $\ssev$ has underlying set $\{1,a,\infty\}$. Its
addition is flat: the sum of two distinct elements is $\infty$.
Multiplication has identity $1$, zero $\infty$, and $a^2=\infty$.
The operations are displayed in Table~\ref{tab:s7}. The symbols $1$
and $\infty$ name elements of this particular algebra; they are not
constant symbols in the signature. Our $\infty$ is the element denoted
by $0$ in the Cayley tables of Yue, Ren and Gao~\cite{YRG}.

\begin{table}[ht]
\centering
\caption{The Cayley tables of $\ssev$.}
\label{tab:s7}
\renewcommand{\arraystretch}{1.12}
\begin{tabular}{c|ccc}
$+$ & $1$ & $a$ & $\infty$ \\
\hline
$1$ & $1$ & $\infty$ & $\infty$ \\
$a$ & $\infty$ & $a$ & $\infty$ \\
$\infty$ & $\infty$ & $\infty$ & $\infty$
\end{tabular}
\qquad
\begin{tabular}{c|ccc}
$\cdot$ & $1$ & $a$ & $\infty$ \\
\hline
$1$ & $1$ & $a$ & $\infty$ \\
$a$ & $a$ & $\infty$ & $\infty$ \\
$\infty$ & $\infty$ & $\infty$ & $\infty$
\end{tabular}
\end{table}

Jackson, Ren and Zhao~\cite[Corollary~5.1]{JRZ} proved that $\ssev$
is, up to isomorphism, the unique nonfinitely based ai-semiring of
order at most three. Wu, Ren and Zhao~\cite{WRZ} proved that
$S_7^0$, obtained by adjoining to $S_7$ an additive identity that is
also a multiplicative zero, is nonfinitely based. Gao, Jackson, Ren
and Zhao~\cite{GJRZ} established sufficient conditions for certain
ai-semiring varieties containing $S_7$ to be nonfinitely based.
The general transfer question is stated in~\cite{GJRZ} and recorded
as Problem~1.2 in~\cite{YRG}:
\begin{quote}
\itshape
Is every finite ai-semiring whose variety contains $S_7$
nonfinitely based?
\end{quote}
Following the terminology used there, a finite ai-semiring $A$ is
\emph{strongly nonfinitely based} if every finite ai-semiring $R$
with $A\in\V(R)$ is nonfinitely based. Our main result is the following.

\begin{theorem}\label{thm:main}
Let $R$ be a finite ai-semiring. If $\ssev\in\V(R)$, then $R$ is
nonfinitely based. Consequently, $\ssev$ is strongly nonfinitely based.
\end{theorem}

The quantification over finite generators is part of this statement.
It is distinct from inherent nonfinite basability, which excludes
membership in every finitely based locally finite variety. Our proof
uses the finiteness of the multiplicative reduct of $R$ to obtain
uniform bounds on its product languages.

The proof has two components. First, we show that identities of
$\ssev$ involving at most $n$ variables preserve polynomials whose
word supports belong to a family of large sets with sufficiently
small pairwise intersections. The key observation is that any small
subfamily has a private point in each of its members. Second, for
every finite semigroup we construct a family with these intersection
properties that realizes every full product using a suitable ordering
of one of its members. The construction uses semilinearity of letter
counts and a finite probabilistic argument. Taking all orderings of
each member produces an identity of $R$ that cannot follow from its
identities in a bounded number of variables.

Section~\ref{sec:prelim} fixes the notation. The substitution argument
is proved in Section~\ref{sec:private}, and the finite semigroup
construction in Section~\ref{sec:products}. Section~\ref{sec:main}
proves Theorem~\ref{thm:main} and gives consequences for power
semirings, including the two power semirings studied in~\cite{YRG}.

\section{Preliminaries}\label{sec:prelim}

Throughout, $\N=\{0,1,2,\ldots\}$ and $[N]=\{1,\ldots,N\}$ for
$N\geq1$. All logarithms are natural. Every algebra has signature
$(+,\cdot)$, with no nullary operations. A finite basis relative to
the ai-semiring axioms exists if and only if a finite basis in this
signature exists, since the ai-semiring axioms themselves form a
finite set.

Let $X$ be a countably infinite set of variables. Write $X^+$ for the
free semigroup on $X$, and $X^*=X^+\cup\{\varepsilon\}$ for the free
monoid. Every ai-semiring term expands, using distributivity, into a
finite nonempty sum of nonempty words. For a term $p$, denote its set
of expanded words by $\Wd(p)$. Recursively,
\begin{equation}\label{eq:expansion}
\Wd(x)=\{x\},\qquad
\Wd(p+q)=\Wd(p)\cup\Wd(q),\qquad
\Wd(pq)=\Wd(p)\Wd(q).
\end{equation}
Addition is represented by union, so repeated summands are discarded.
These finite nonempty word sets form the free ai-semiring;
see Ku\v{r}il and Pol\'{a}k~\cite[Theorem~2.5]{KP}.
In particular, the ai-semiring axioms preserve $\Wd(p)$.

For $u\in X^+$, let $\occ(x,u)$ be the number of occurrences of $x$
in $u$, and let $\supp(u)$ be its set of variables. The word $u$ is
\emph{linear} if $\occ(x,u)\leq1$ for every $x$. Its
\emph{commutative image} is
\[
\cm(u)=\prod_{x\in\supp(u)}x^{\occ(x,u)}.
\]
Thus $u$ is linear precisely when $\cm(u)$ is squarefree. We also
write $\cm(p)$ for the formal commutative polynomial obtained by
applying $\cm$ to all words in $\Wd(p)$ and discarding repetitions.
This notation records multiplicities of variables, but not their
order. It does not impose commutativity on an ai-semiring under
consideration. We use $\cm(\varepsilon)=1$ solely as notation for an
absent multiplicative context.

The natural order on an ai-semiring is
\[
b\preccurlyeq c\quad\Longleftrightarrow\quad b+c=c.
\]
For terms, $p\preccurlyeq q$ abbreviates the identity $p+q\approx q$.
In $\ssev$ the only element below $1$ is $1$, the only element below
$a$ is $a$, and every element is below $\infty$. A nonempty word
evaluated on $\{1,a\}$ has value $a$ exactly when it has one
occurrence assigned $a$; two or more such occurrences give $\infty$.

Write $\Id(A)$ for the identities of $A$, and $\Id_n(A)$ for those
whose two sides together involve at most $n$ variables. We use the
standard substitution and context formulation of equational
deduction; see~\cite{BS} or~\cite[Lemma~2.2]{YRG}. Namely, a
consequence is obtained by a finite sequence of replacements of a
substituted side of an identity by its other side, in either
direction and inside a term context. A substitution sends each
variable to a nonempty term, and hence to a nonempty polynomial of
nonempty words. Empty words are allowed only as omitted left or
right contexts, never as substitution values.

\section{Private points and identities in boundedly many variables}
\label{sec:private}

For a finite nonempty set $E\subseteq X$, put
$m_E=\prod_{x\in E}x$, a squarefree commutative monomial. If $\E$
is a finite nonempty family of such sets, put
$p_{\E}=\sum_{E\in\E}m_E$. Since multiplication in $\ssev$ is
commutative, these polynomials have unambiguous evaluations in
$\ssev$. The \emph{private points} of $E$ in $\E$ form the set
\[
D_E=E\setminus\bigcup_{F\in\E\setminus\{E\}}F.
\]
A set $Z\subseteq\bigcup\E$ is an \emph{exact transversal} of $\E$
if $|Z\cap E|=1$ for every $E\in\E$.
Exact transversals occur in the description of the identities of
$\ssev$ by Jackson, Ren and Zhao~\cite[Proposition~5.5]{JRZ};
see also~\cite[Lemma~1.2]{GJRZ}. The next lemma gives a direct
rigidity argument for families with private points.

\begin{lemma}\label{lem:private}
Suppose that $D_E\neq\varnothing$ for every $E\in\E$. For every
nonempty commutative monomial $m$,
\begin{equation}\label{eq:rigidity}
\ssev\models m\preccurlyeq p_{\E}
\quad\Longleftrightarrow\quad
m=m_E\text{ for some }E\in\E.
\end{equation}
Moreover, every point of $\bigcup\E$ belongs to an exact transversal
of $\E$.
\end{lemma}

\begin{proof}
The reverse implication in~\eqref{eq:rigidity} follows from
idempotent addition. Suppose that the inequality holds. If $m$
contained a variable outside $\bigcup\E$, assigning that variable
$\infty$ and all other variables $1$ would give $m=\infty$ and
$p_{\E}=1$, a contradiction. Thus $\supp(m)\subseteq\bigcup\E$.

Choose one point $z_E\in D_E$ independently for each $E\in\E$.
Assign $a$ to the chosen points and $1$ to all other variables.
Every monomial $m_E$ has value $a$, so $p_{\E}=a$ and the assumed
inequality forces $m=a$. Consequently,
\begin{equation}\label{eq:private-exponents}
\sum_{E\in\E}\occ(z_E,m)=1
\end{equation}
for every such choice. Varying $z_E$ with the other choices fixed
shows that the exponent of each point of $D_E$ is a common
nonnegative integer $\alpha_E$. By~\eqref{eq:private-exponents},
$\sum_{E\in\E}\alpha_E=1$. Hence there is a unique $E_*\in\E$
with $\alpha_{E_*}=1$, and all other $\alpha_E$ are zero.

Now let $x\in\bigcup\E$ be a point that is not private to any
member. Take $x$, together with one private point from each member
not containing $x$. This is an exact transversal. Assigning $a$ to
it and $1$ elsewhere gives
\[
\occ(x,m)+\sum_{\substack{E\in\E\\x\notin E}}\alpha_E=1.
\]
It follows that $\occ(x,m)=1$ if $x\in E_*$ and is zero otherwise.
Together with the exponents on the private points, this proves
$m=m_{E_*}$.

For the last assertion, the same construction works for any
$x\in\bigcup\E$, whether or not it is private: take $x$ and one
private point from each member not containing it.
\end{proof}

Call a term $p$ \emph{supported by $\F$} if every word in $\Wd(p)$
is linear and has support in $\F$. Different words may have the
same support.

\begin{lemma}\label{lem:stability}
Let $n,k\geq1$ and $\ell\geq0$ be integers. Let $\F$ be a finite
nonempty family of finite sets of variables such that
\begin{equation}\label{eq:intersection-condition}
|E|\geq k,\qquad |E\cap F|\leq\ell\quad(E\neq F),
\qquad k>(2^n-2)\ell.
\end{equation}
If $p$ is supported by $\F$ and $p\approx q$ follows from
$\Id_n(\ssev)$ and the ai-semiring axioms, then $q$ is supported
by $\F$.
\end{lemma}

\begin{proof}
We first prove the assertion for one replacement inside a fixed
multiplicative context. Let $f\approx g$ be an identity of $\ssev$
in at most $n$ variables, let $\phi$ be an arbitrary substitution,
and let $l,r\in X^*$. Suppose that $l\phi(f)r$ is supported by
$\F$. Put $c=\cm(lr)$. We shall prove that $l\phi(g)r$ is also
supported by $\F$.

The two sides of any identity of $\ssev$ have the same variable
content. Indeed, assigning $\infty$ to a variable occurring on only
one side and $1$ to all others gives values $\infty$ and $1$.
Let $y_1,\ldots,y_d$, with $d\leq n$, be this common content.

Every word in $\Wd(f)$ is linear. Otherwise some $y_i$ occurs
twice in such a word. Choosing the same nonempty word from
$\Wd(\phi(y_i))$ at those occurrences, and choosing arbitrary
words at the remaining occurrences, produces a non-linear word
in $\Wd(l\phi(f)r)$, a contradiction. In particular, $f$ has at
most $2^d-1\leq2^n-1$ distinct commutative images of words.

Fix one commutative monomial $m_i$ of $\cm(\phi(y_i))$ for each
$i$. All monomials of
\begin{equation}\label{eq:fixed-substitution}
c\,\cm(f)(m_1,\ldots,m_d)
\end{equation}
are squarefree and correspond to members of $\F$. Their distinct
supports form a nonempty subfamily $\E\subseteq\F$ with
$|\E|\leq2^n-1$. For every $E\in\E$,
\[
\left|E\cap\bigcup_{F\in\E\setminus\{E\}}F\right|
\leq(|\E|-1)\ell\leq(2^n-2)\ell<k\leq|E|.
\]
Thus every member of $\E$ has a private point. This observation
holds for every fixed choice of the monomials $m_i$.

We next show that every word in $\Wd(g)$ is linear. For this
purpose keep the current choice of the $m_i$. Fix $i$ and take
$z\in\supp(m_i)$, which is possible because substitution words
are nonempty. Since $y_i$ occurs in $f$, the product $cm_i$
divides a squarefree monomial in~\eqref{eq:fixed-substitution}.
Hence $z\notin\supp(c)$, and $m_i$ is squarefree.

There is an exact transversal $Z$ of $\E$ that contains $z$ and
avoids $\supp(c)$. If $\E$ has one member, take $Z=\{z\}$.
Otherwise $\supp(c)$ is contained in every member of $\E$, so
no private point belongs to $\supp(c)$. The construction at the
end of Lemma~\ref{lem:private} then gives the required $Z$.
Assign $a$ to variables in $Z$ and $1$ to all other variables.
Under this evaluation $c=1$ and every monomial
in~\eqref{eq:fixed-substitution} has value $a$.
Each $m_j$ divides at least one of these squarefree monomials,
so its value lies in $\{1,a\}$; the value of $m_i$ is $a$.

The values of the $m_j$ therefore define an assignment of
$y_1,\ldots,y_d$ under which $f=a$ and $y_i=a$. Since
$\ssev\models f\approx g$, this assignment also gives $g=a$.
If any word of $g$ repeated $y_i$, its value would be $\infty$,
and consequently $g$ would have value $\infty$. This is
impossible. Applying this argument for each $i$ proves that all
words of $g$ are linear.

Now take an arbitrary word in $\Wd(l\phi(g)r)$, and let $m$ be
its commutative image. It comes from a word $v\in\Wd(g)$ and
choices of words at its variable occurrences. Since $v$ is
linear, these choices give one monomial $m_i$ for each variable
appearing in $v$; choose $m_i$ arbitrarily for the other variables.
Thus
\[
m=c\,\cm(v)(m_1,\ldots,m_d).
\]
With this new fixed choice, the monomials of
$c\,\cm(f)(m_1,\ldots,m_d)$ again form $p_{\E}$ for a
subfamily $\E$ having a private point in every member. Because
$v$ is a summand of $g$ and $\ssev\models f\approx g$, we have
$\ssev\models v\preccurlyeq f$. Substitution and multiplication
by the context give
\[
\ssev\models
c\,\cm(v)(m_1,\ldots,m_d)
\preccurlyeq
c\,\cm(f)(m_1,\ldots,m_d)=p_{\E}.
\]
Lemma~\ref{lem:private} implies that $m=m_E$ for some
$E\in\E\subseteq\F$. This proves the assertion for the fixed
multiplicative context.

Finally, a term context with a single occurrence of a placeholder
$z$ expands as
\[
h+\sum_{j=1}^s l_jzr_j,
\]
where $s\geq1$, $l_j,r_j\in X^*$, and $h$ may be absent.
This follows by induction on the context using distributivity.
If the old expanded term is supported by $\F$, every word in
each $l_j\phi(f)r_j$ has that property. The argument above
applies to each $j$, and the summands of $h$ are unchanged.
It also applies with $f$ and $g$ interchanged. The ai-semiring
axioms preserve expanded word sets by~\eqref{eq:expansion}.
Induction on a sequence of equational replacements proves the
lemma.
\end{proof}

\begin{remark}\label{rem:stability}
The members of $\F$ need not have equal sizes. Moreover, the
proof permits arbitrary polynomial substitutions. The argument
establishing linearity of $g$ is what allows each output word to
be represented by a consistent choice of one substitution word
for each source variable.
\end{remark}

\section{Sparse realizations of products in finite semigroups}
\label{sec:products}

Let $S=\{s_1,\ldots,s_r\}$ be a finite semigroup, viewed also as
an alphabet. For $b\in S$, define
\[
L_b=\{u\in S^+ : \text{the product of the letters of $u$ in $S$ is $b$}\}.
\]
The \emph{Parikh vector} of $u\in S^*$ is the vector
$\Psi(u)=(|u|_{s_1},\ldots,|u|_{s_r})\in\N^r$ of its letter
counts. Put $P_b=\Psi(L_b)$.

A subset of $\N^r$ is \emph{linear} if it has the form
\begin{equation}\label{eq:linear-set}
\vect b+\N V
=\left\{\vect b+\sum_{j=1}^q n_j\vect v_j:n_j\in\N\right\},
\qquad V=\{\vect v_1,\ldots,\vect v_q\}\subseteq\N^r.
\end{equation}
The period set $V$ may be empty. A finite union of linear sets
is \emph{semilinear}. The semilinearity of Parikh images of regular
languages is a special case of Parikh's theorem~\cite{Parikh}.
We include an elementary regular-language argument for the property
needed below; background on regular expressions can be found
in~\cite{Conway}.

\begin{lemma}\label{lem:semilinear}
For every $b\in S$, the set $P_b$ is semilinear. Finite
representations of all these sets can be chosen together with
an integer $B\geq1$ such that every component
\eqref{eq:linear-set} in every representation satisfies
\begin{equation}\label{eq:B}
\nnorm{\vect b}+\sum_{j=1}^q\nnorm{\vect v_j}\leq B.
\end{equation}
The integer $B$ depends only on $S$.
\end{lemma}

\begin{proof}
The language $L_b$ is recognized by a finite automaton whose
states are a new initial state $*$ and the elements of $S$.
Reading $s_i$ at $*$ leads to $s_i$, and reading $s_i$ at a
state $t\in S$ leads to $ts_i$. The only accepting state is
$b$. The initial state is separate even when $S$ has an
identity, so the empty word is not accepted.

For completeness, a finite automaton yields a regular expression
by eliminating states. If $R_{ij}^{(h)}$ denotes an expression
for paths from state $i$ to state $j$ whose intermediate states
are among the first $h$ states, then
\[
R_{ij}^{(h)}=R_{ij}^{(h-1)}\cup
R_{ih}^{(h-1)}\bigl(R_{hh}^{(h-1)}\bigr)^*R_{hj}^{(h-1)}.
\]
Start with the direct transition labels and include the empty
word when $i=j$. This gives a finite regular expression for
$L_b$.

Parikh vectors turn concatenation into vector addition.
Semilinear sets are closed under finite unions and setwise
addition. They are also closed under the additive star, the
operation taking all finite sums of elements, including the
empty sum. Indeed, if
$A=\bigcup_{i=1}^h(\vect b_i+\N V_i)$, then
\begin{equation}\label{eq:star}
A^*=\bigcup_{I\subseteq[h]}
\left(
\sum_{i\in I}\vect b_i+
\N\left(\{\vect b_i:i\in I\}\cup\bigcup_{i\in I}V_i\right)
\right).
\end{equation}
The summand for $I=\varnothing$ is $\{\vect0\}$. To verify
the formula, let $I$ record the linear components used at least
once in a finite sum. Their first base vectors give
$\sum_{i\in I}\vect b_i$, and further base vectors and all
period vectors give the indicated nonnegative combinations.
Conversely, any such combination can be realized by using
component $i$ once plus as many additional times as its base
vector requires, and placing all its period vectors in one
of those occurrences.

The Parikh images of the atomic regular expressions are
semilinear. Induction on a regular expression, using the three
closure properties just proved, shows that $P_b$ is semilinear.
There are finitely many $b\in S$ and finitely many components
in their chosen representations. Taking the maximum of their
quantities in~\eqref{eq:B}, and of $1$, gives the required $B$.
\end{proof}

Fix representations and $B$ as in Lemma~\ref{lem:semilinear}.
For an assignment $\nu:[N]\to S$, say that a nonempty set $E\subseteq[N]$
\emph{realizes its full product} if some ordering
$i_1,\ldots,i_s$ of the distinct elements of $E$ satisfies
\begin{equation}\label{eq:realizes}
\nu(i_1)\cdots\nu(i_s)=\nu(1)\cdots\nu(N).
\end{equation}
The ordering is allowed to depend on the assignment.

\begin{lemma}\label{lem:probability}
Let $k\geq1$ be an integer, and put
\begin{equation}\label{eq:constants}
L=k+B,\qquad K=k+2B,\qquad C=rB,\qquad
c=\frac{1}{(2B+1)(4L)^K(K!)^r}.
\end{equation}
For every integer $N\geq4LK$, choose a random subset $E\subseteq[N]$ by first
choosing its size uniformly from $\{k,k+1,\ldots,K\}$ and then
choosing uniformly among subsets of that size. For every
assignment $\nu:[N]\to S$,
\begin{equation}\label{eq:probability}
\Pr\bigl(E\text{ realizes the full product of }\nu\bigr)
\geq cN^{-C}.
\end{equation}
In particular, the exponent $C$ is independent of $k$.
\end{lemma}

\begin{proof}
Let $b=\nu(1)\cdots\nu(N)$, and let
$\vect N=(N_1,\ldots,N_r)$ be the count vector of the assignment.
Then $\vect N\in P_b$, so it belongs to one of the chosen
linear components:
\[
\vect N=\vect b+\sum_{j=1}^q n_j\vect v_j.
\]
Put $\alpha=L/N\leq1$ and define
\begin{equation}\label{eq:rounding}
\vect M=\vect b+\sum_{j=1}^q\lfloor\alpha n_j\rfloor\vect v_j.
\end{equation}
Since $0\leq\lfloor\alpha n_j\rfloor\leq n_j$, we have
$\vect0\leq\vect M\leq\vect N$ coordinatewise. Moreover,
$\vect M\in P_b$ and
\begin{align*}
\nnorm{\vect M-\alpha\vect N}
&\leq(1-\alpha)\nnorm{\vect b}
  +\sum_{j=1}^q\nnorm{\vect v_j}\\
&\leq B.
\end{align*}
Because $\nnorm{\alpha\vect N}=L$, the integer
$s=\nnorm{\vect M}$ lies between $k=L-B$ and $K=L+B$.

There is a word over $S$ with count vector $\vect M$ and product
$b$. Any set of positions with exactly $M_i$ occurrences of
$s_i$ can be ordered to spell this word: match its occurrences
of each letter to the chosen positions with that value.
Thus every set with count vector $\vect M$ realizes the full
product. Conditional on $|E|=s$, the probability that $E$
has exactly this count vector is
\begin{equation}\label{eq:hypergeometric}
\frac{\prod_{i=1}^r\binom{N_i}{M_i}}{\binom Ns}.
\end{equation}
Consequently, this expression is a lower bound for the
conditional probability that $E$ realizes the full product.

Call an index $i$ rare if $N_i<N/(2L)$. The bound on the
rounding error gives
\[
M_i\leq\alpha N_i+B<B+\tfrac12
\]
for every rare $i$. Hence $M_i\leq B$, and the number of
selected positions with rare values is
$d=\sum_{i\text{ rare}}M_i\leq rB=C$.
For a nonrare $i$, the hypothesis $N\geq4LK$ implies
\mbox{$N_i\geq2K\geq2M_i$}. Consequently,
\[
\binom{N_i}{M_i}
=\frac{N_i(N_i-1)\cdots(N_i-M_i+1)}{M_i!}
\geq\frac{(N/(4L))^{M_i}}{M_i!},
\]
with the usual value $1$ when $M_i=0$.
For rare indices use $\binom{N_i}{M_i}\geq1$.
Together with $\binom Ns\leq N^s$, this bounds
\eqref{eq:hypergeometric} below by
\[
\frac{N^{s-d}}{(4L)^{s-d}(K!)^rN^s}
\geq (4L)^{-K}(K!)^{-r}N^{-C}.
\]
Finally, the probability of choosing the size $s$ is
$1/(2B+1)$. This proves~\eqref{eq:probability}.
\end{proof}

We now impose coverage and small pairwise intersections
simultaneously by independent sampling and the union bound.
These are standard tools of the probabilistic method;
see Alon and Spencer~\cite{AS} for general background.

\begin{lemma}\label{lem:cover}
For every finite semigroup $S$ and every integer $n\geq1$,
there exist integers $N>K\geq k\geq1$ and $\ell\geq0$, and a
finite nonempty family $\F\subseteq\Pow([N])$, such that
\begin{enumerate}
\item $k\leq|E|\leq K$ for every $E\in\F$;
\item $|E\cap F|\leq\ell$ for distinct $E,F\in\F$, and
      $k>(2^n-2)\ell$;
\item for every assignment $\nu:[N]\to S$, some $E\in\F$
      realizes its full product.
\end{enumerate}
\end{lemma}

\begin{proof}
Use $r$, $B$ and $C=rB$ as above. Choose, in this order,
\begin{equation}\label{eq:choose-k}
t=2C+3,\qquad \ell=t-1,\qquad
k=\max\{t,(2^n-2)(t-1)+1\}.
\end{equation}
Let $L,K,c$ be as in~\eqref{eq:constants}, and put
\[
A=\frac{\log r+2}{c},\qquad
M=\lceil AN^{C+1}\rceil.
\]
Independently sample $M$ sets using the distribution in
Lemma~\ref{lem:probability}.

For a fixed assignment $\nu$, the probability that no sampled
set realizes its full product is at most
\[
(1-cN^{-C})^M\leq \exp(-cN^{-C}M)
\leq\exp(-cAN).
\]
There are $r^N$ assignments. The union bound therefore shows
that the probability that condition~\textup{(iii)} fails is at most
\begin{equation}\label{eq:coverage-failure}
r^N\exp(-cAN)=\exp(-2N).
\end{equation}

We also estimate intersections of two independently sampled
sets $E$ and $F$. Given $E$ and $|F|=s$, for a fixed
$t$-element subset $T\subseteq E$,
\[
\Pr(T\subseteq F\mid E,|F|=s)
=\frac{\binom{N-t}{s-t}}{\binom Ns}
\leq\left(\frac{K}{N-t+1}\right)^t.
\]
Since $|E|\leq K$, another union bound, followed by averaging
over the sizes, gives for $N\geq2t$
\begin{equation}\label{eq:pair-failure}
\Pr(|E\cap F|\geq t)\leq b_0N^{-t},
\qquad b_0=\binom Kt(2K)^t.
\end{equation}
It follows that the probability of any pair of sampled sets
having intersection of size at least $t$ is at most
\begin{equation}\label{eq:all-pairs-failure}
\binom M2 b_0N^{-t}
\leq \frac{b_0(A+1)^2}{2N},
\end{equation}
where we used $M\leq(A+1)N^{C+1}$ and $t=2C+3$.

The quantities $r,B,C,t,\ell,k,L,K,c,A,b_0$ are independent
of $N$. Choose an integer $N$ with
\[
N\geq\max\{4LK,2t,K+1,2\},\qquad N>b_0(A+1)^2.
\]
Then~\eqref{eq:coverage-failure} and
\eqref{eq:all-pairs-failure} are each less than $1/2$.
With positive probability neither failure occurs. The resulting
sets are all distinct: a repeated set would have intersection
at least $k\geq t$ with its other occurrence. Taking them as
$\F$ proves~\textup{(i)} and~\textup{(iii)}, and the choices
in~\eqref{eq:choose-k} give~\textup{(ii)}.
\end{proof}

\begin{remark}\label{rem:order}
Semilinearity gives a short count vector representing the same
product, but does not preserve the original order of the selected
positions. Accordingly, condition~\textup{(iii)} allows all
orderings of a member of $\F$. Also, the interval of permitted
sizes $[k,k+2B]$ accommodates the rounding error in
\eqref{eq:rounding}; an exact prescribed size is not required.
\end{remark}

\section{The main theorem and power semirings}\label{sec:main}

For a finite nonempty set $E$ of variables, let $\Sym(E)$ be
the set of all words that use each variable of $E$ exactly once.

\begin{proof}[Proof of Theorem~\ref{thm:main}]
Suppose, to the contrary, that $R$ has a finite identity basis
$\Sigma$ relative to the ai-semiring axioms. Choose $n\geq1$
at least as large as the number of variables in any identity of
$\Sigma$. Since $\ssev\in\V(R)$, every identity of $R$ holds
in $\ssev$, and therefore
\[
\Sigma\subseteq\Id_n(\ssev).
\]

Apply Lemma~\ref{lem:cover} to the multiplicative semigroup
$(R,\cdot)$ and to $n$. Identify a position $i\in[N]$ with a
distinct variable $x_i$, and use the same notation $\F$ for
the resulting family of sets of variables. Define
\begin{equation}\label{eq:obstruction}
P=\sum_{E\in\F}\ \sum_{u\in\Sym(E)}u,
\qquad w=x_1x_2\cdots x_N.
\end{equation}
For every assignment of these variables to $R$,
Lemma~\ref{lem:cover}\textup{(iii)} supplies a member $E$ and
an ordering of it whose product equals the value of $w$.
That ordering occurs as a summand of $P$. Idempotent addition
then gives
\begin{equation}\label{eq:witness}
R\models P+w\approx P.
\end{equation}

The term $P$ is supported by $\F$. The intersection and size
conditions in Lemma~\ref{lem:cover} are exactly the hypotheses
of Lemma~\ref{lem:stability}. Thus every term equated to $P$
using $\Sigma$ and the ai-semiring axioms must also be
supported by $\F$. However, $P+w$ has the word $w$ with
support of size $N>K$. This support is not a member of $\F$.
Hence~\eqref{eq:witness} cannot follow from $\Sigma$,
contradicting that $\Sigma$ is an identity basis of $R$.
\end{proof}

The construction gives a slightly more precise separation
statement, independently of the variety-containment hypothesis.

\begin{corollary}\label{cor:separation}
For every finite ai-semiring $R$ and every $n\geq1$, there is
an identity of $R$ that is not a consequence of
$\Id_n(\ssev)$ together with the ai-semiring axioms.
\end{corollary}

\begin{proof}
Construct $P$ and $w$ from the multiplicative semigroup of $R$
as in~\eqref{eq:obstruction}. The identity $P+w\approx P$
holds in $R$ by Lemma~\ref{lem:cover}, and
Lemma~\ref{lem:stability} excludes its derivation from
$\Id_n(\ssev)$.
\end{proof}

We now give consequences for the powerset construction.
For a semigroup $S$, let $\Pow(S)$ be its full power semiring
and let $\Pow^+(S)$ be its nonempty power semiring. Addition
is union and multiplication is setwise multiplication:
\[
A+B=A\cup B,\qquad AB=\{ab:a\in A,\ b\in B\}.
\]
Thus $\Pow^+(S)$ is a subalgebra of $\Pow(S)$. When $A$ is
an ai-semiring, the notation $\Pow(A,\cdot)$ and
$\Pow^+(A,\cdot)$ specifies that the powerset construction is
applied to its multiplicative reduct.

\begin{proposition}\label{prop:upward}
If $A$ is a finite ai-semiring with $\ssev\in\V(A)$, then
$A$ is strongly nonfinitely based.
\end{proposition}

\begin{proof}
If $T$ is finite and $A\in\V(T)$, then
$\ssev\in\V(A)\subseteq\V(T)$. By Theorem~\ref{thm:main},
$T$ is nonfinitely based, as required.
\end{proof}

\begin{proposition}\label{prop:powers}
Let $A$ be a finite ai-semiring. The map
\[
\pi:\Pow^+(A,\cdot)\longrightarrow A,
\qquad U\longmapsto\sum_{a\in U}a
\]
is a surjective homomorphism. Consequently, if
$\ssev\in\V(A)$, both $\Pow^+(A,\cdot)$ and
$\Pow(A,\cdot)$ are strongly nonfinitely based.
\end{proposition}

\begin{proof}
The map is surjective because $\pi(\{a\})=a$.
Associativity, commutativity and idempotence of addition give
$\pi(U\cup V)=\pi(U)+\pi(V)$. Distributivity and idempotence
give
\[
\pi(U)\pi(V)=\sum_{u\in U,\ v\in V}uv
=\sum_{b\in UV}b=\pi(UV).
\]
It follows that
\[
A\in\V\bigl(\Pow^+(A,\cdot)\bigr)
\subseteq\V\bigl(\Pow(A,\cdot)\bigr).
\]
Both power semirings are finite, so the conclusion follows
from Proposition~\ref{prop:upward}.
\end{proof}

\begin{corollary}\label{cor:s7powers}
Both $\Pow^+(\ssev,\cdot)$ and $\Pow(\ssev,\cdot)$ are
strongly nonfinitely based.
\end{corollary}

\begin{proof}
Apply Proposition~\ref{prop:powers} with $A=\ssev$.
\end{proof}

Yue, Ren and Gao~\cite{YRG} prove that these two power
semirings are nonfinitely based and give explicit infinite
bases for their identities. Corollary~\ref{cor:s7powers}
strengthens their nonfinite basis conclusions to strong
nonfinite basability.

\begin{corollary}\label{cor:groups}
Let $G$ be a finite group of order at least three. Then
both $\Pow^+(G)$ and $\Pow(G)$ are strongly nonfinitely based.
\end{corollary}

\begin{proof}
Let $e$ be the identity of $G$, and put $D=G\setminus\{e\}$.
We claim that $D^2=G$. Given $g\in G$, choose
$x\in G\setminus\{e,g\}$, which is possible since $|G|\geq3$.
Then $x\in D$ and $x^{-1}g\in D$, so
$g=x(x^{-1}g)\in D^2$.

The three distinct nonempty sets $\{e\}$, $D$ and $G$ form a
subsemiring of $\Pow^+(G)$. Their union of any two distinct
members is $G$. Setwise multiplication has identity $\{e\}$,
satisfies $D^2=G$, and has $G$ as an absorbing element on
this three-element subsemiring. Thus
\[
1\longmapsto\{e\},\qquad a\longmapsto D,
\qquad\infty\longmapsto G
\]
is an embedding of $\ssev$. Proposition~\ref{prop:upward}
applies to both $\Pow^+(G)$ and $\Pow(G)$.
\end{proof}

Gao, Ren, Shao and Yue~\cite{GRSY} establish that the nonempty
power semiring of a finite group is nonfinitely based exactly
when the group has order at least three. Corollary~\ref{cor:groups}
gives the strong nonfinite basis property in that range and
the same property for the full power semiring.

\end{document}